\documentclass[12pt,letterpaper]{amsart}
\usepackage{ amsmath,amsthm, amsbsy,amsfonts,amssymb, txfonts}
\usepackage{stmaryrd,mathrsfs,graphicx, amscd, tikz-cd, tikz, cancel}
\usepackage{tikz}
\usetikzlibrary{matrix,arrows,decorations.pathmorphing}

\usepackage[active]{srcltx}
\allowdisplaybreaks
\numberwithin{equation}{section}

\usepackage[
	hypertexnames=false,
	hyperindex,
	pagebackref,
	pdftex,
	breaklinks=true,
	bookmarks=false,
	colorlinks,
	linkcolor=blue,
	citecolor=red,
	urlcolor=red,
]{hyperref}
\usepackage{hyperref}

\def\CC{{\mathbb C}}  
  
\def\EE{{\mathbb E}} 
 
\def\GG{{\mathbb G}}

\def\MM{{\mathbb M}}

\def\QQ{{\mathbb Q}} 
\def\RR{{\mathbb R}}

\def\ssm{\smallsetminus}

\def\G{\Gamma}

\def\bs{\backslash}
\newcommand{\eps}{\varepsilon}

\newcommand{\p}{\partial}

\def\Acal{{\mathcal A}}
\def\Ascr{{\mathscr A}}

\def\Ecal{{\mathcal E}} 
\def\Fcal{{\mathcal F}} 
\def\Gcal{{\mathcal G}}

\def\Iscr{{\mathscr I}} 
 
\def\Kcal{{\mathcal K}}
\def\Lcal{{\mathcal L}}
\def\Lscr{{\mathscr L}}
\def\Mcal{{\mathcal M}}

\def\Ocal{{\mathcal O}}
\def\Pcal{{\mathcal P}}

\def\Rcal{{\mathcal R}}

\def\Ucal{{\mathcal U}}  
\def\Vcal{{\mathcal V}}

\def\Cscr{{\mathscr C}}  
\def\Dscr{{\mathscr D}}

\def\la{\langle}
\def\ra{\rangle}
\def\half{{\tfrac{1}{2}}}

\def\pt{{\scriptscriptstyle\bullet}}

\newcommand\Gr{\operatorname{Gr}}
\newcommand\Hl{\operatorname{H}}

\newcommand\Hom{\operatorname{Hom}}

\newcommand\IH{\operatorname{IH}}

\newcommand\lie{\operatorname{Lie}}

\newcommand\MF{\operatorname{MF}}

\newcommand\Star{\operatorname{Star}}

\newtheorem{theorem}{Theorem}[section]
\newtheorem{lemma}[theorem]{Lemma}
\newtheorem{proposition}[theorem]{Proposition}

\newtheorem{definition}{Definition}\numberwithin{definition}{section}

\theoremstyle{remark}

\newtheorem{observation}[theorem]{Observation}
\newtheorem{remark}[theorem]{Remark}

\title[$L_2$ Hodge structure]{ Hodge decomposition  of  $L_2$-cohomology and intersection cohomology 
of  a Shimura variety}

\author{Eduard Looijenga}
\address{Mathematics Department, University of Chicago (USA) and Mathematisch Instituut, Universiteit Utrecht (Nederland)}
\email{e.j.n.looijenga@uu.nl}

\begin{document}


\maketitle
\begin{abstract}
Classical Hodge theory  endows the square integrable  cohomology of a Shimura variety $X$ with values in 
a locally homogeneous  polarized variation of Hodge structure $\EE$ with a natural Hodge decomposition.
The theory of Morihiko Saito does the same for the $\EE$-valued  intersection  cohomology of its Baily-Borel compactification. Existing  proofs  of the Zucker conjecture identify these two, but do not claim this for their Hodge decompositions. We show that the proofs given in  \cite{looij:l2} and \cite{lr} yields that as well.
\end{abstract}

\section*{Introduction}
The Zucker conjecture states that for a Shimura variety $X$ its square integrable (=$L_2$) cohomology is naturally isomorphic with 
the intersection cohomology of its Baily-Borel compactification $X^*$ and that this is even  true if instead of constant coefficients 
we let these cohomology groups take their values in a  locally homogeneous polarized variation of Hodge structure: $\Hl^\pt_{(2)}(X,\EE)\cong \IH^\pt(X,\EE)$.
This conjecture was settled a long time ago (\cite{looij:l2}, \cite{ss}, \cite{lr}), yet one aspect remained open. 
Each side of the isomorphism comes with its Hodge decomposition: the representation of $L_2$-cohomology 
by harmonic  forms define a Hodge decomposition of $\Hl^\pt_{(2)}(X,\EE)$ and the theory of Morihiko Saito 
\cite{saito:hmodules} puts one on $\IH^\pt(X,\EE)$, but it was not established that under this isomorphism  
the two coincide.  Harris and Zucker, who  stated this as a conjecture in \cite{hz}, 5.3,  proved  that the natural map  
$\Hl^\pt_{(2)}(X,\EE)\to \Hl^\pt(X, \EE)$ is in fact a 
morphism of mixed Hodge structures (\cite{hz}, Thm.\ 5.4). The purpose of this note is to show that 
$\Hl^\pt_{(2)}(X,\EE)\cong \IH^\pt(X,\EE)$ is  indeed an isomorphism of Hodge structures.
 
The proofs of the Zucker conjecture  mentioned above took  a local form, namely as an equality in the derived category 
of bounded below complexes of $\CC$-vector spaces on $X^*$. The refinement which we prove  here  is of a similar 
nature, the derived category in question now being  one of filtered complexes that encode in a local manner the 
Hodge filtration. This implies that  the cohomology sheaves of successive quotients will be coherent $\Ocal_{X^*}$-modules, rather than being  locally constant. Our main tool are  the local Hecke operators that were  introduced in  
\cite{looij:l2} and further developed in \cite{lr} with the express goal to prove the Zucker conjecture as stated above. 
Indeed, we will show   that  a minor modification  can do the same job in a filtered setting.  
This could well have been done at the time, because the underlying technique was then already available.
This also explains why much of this paper consists of revisiting and recalling work from that  era.

Since  we  posted  the first version of this paper, Mingyu Ni posted another proof of our main result  \cite{ni}. This preprint  also contains 
several other related results of interest.

\section{Hodge modules}
\subsection{Polarizable Hodge modules} Let $X$ be a complex  manifold  and $\EE$ a polarizable variation of  Hodge structures 
on $X$ of weight $w$. We here think of $\EE$ as local system of 
finite dimensional  $\QQ$-vector spaces which contains a lattice (this guarantees that the monodromy can 
be given by integral matrices and is quasi-unipotent when restricted to a punctured disk), 
but we do not consider that lattice as part of the data. Recall that the underlying holomorphic vector bundle  
$\Ecal:=\Ocal_X\otimes_\QQ \EE$  then comes with a flag of holomorphic subbundles  $\Fcal^\pt\Ecal$ 
(the Hodge filtration) satisfying  Griffiths transversality, i.e., the property  that  the flat connection 
$\nabla=d\otimes_\QQ 1_\EE$ on $\Ecal$ takes 
$\Fcal^p\Ecal$ to $\Omega^1_X\otimes_{\Ocal_X}\Fcal^{p-1}\Ecal$. This connection extends to
a derivation of degree 1 of square  zero  in $\Omega^\pt(\Ecal)=\Omega^\pt_X\otimes_{\Ocal_M}\Ecal=\Omega^\pt_M\otimes_\QQ \EE$  and furnishes  the holomorphic De Rham resolution of $\EE_\CC:=\CC\otimes_\QQ \EE$,
\[
\EE_\CC\to (\Omega^\pt_X(\Ecal), \nabla).
\]
Note that this resolution  comes with a  (Hodge) filtration $\Fcal^\pt\Omega^\pt_X(\Ecal)$ by subcomplexes whose $p$th term is $\sum_{r+s\ge p}\Omega^r_X(\Fcal^s\Ecal)$.  

We can instead consider  $(\Ecal, \Fcal^\pt\Ecal)$ as a filtered $\Dscr_X$-module,  
where $\Dscr_X$ stands for the sheaf of holomorphic differential operators $\Ocal_X\to \Ocal_X$ (itself  filtered by order). 
It thus determines an object of a derived category of filtered $\Dscr_X$-modules: this  is the  
`Riemann-Hilbert incarnation' of  the pair $(\Fcal^\pt\Ecal, \nabla)$. 
If we ignore the filtration, then both represent $\EE_\CC$ in  $D_c^b(X,\CC)$, 
the derived category of constructible $\CC_X$-modules  with bounded cohomology. 

The  Dolbeault resolution gives  a fine resolution (meaning that the sheaves  admit partitions of unity relative to any given open cover) of the  Hodge filtered     complex  $(\Omega^\pt_X(\Ecal), \nabla)$ by the smooth bigraded de Rham complex  
\[
(\Ascr^{\pt,\pt}(\EE_\CC), d=\p +\bar\p)
\]
  with its Hodge filtration defined in an obvious manner. 
So if $U\subset X$ is open, then this puts on $\Hl^k(U; \EE)$ a filtration $F^\pt\Hl^k(U; \EE)$, with $F^p\Hl^k(U; \EE)$
the part representable by a $d$-closed section of $\Fcal^p\Ascr_X^{\pt, \pt}(\EE)$ over $U$. 
The meaning of this filtration is however usually obscure unless  $U$ is quite special.

This illustrates  the notion of a \emph{(polarized) Hodge module} in the analytic setting  in its most basic  form 
(see \cite{saito:hmodules}, \cite{saito:guide}), provided we make a degree  shift by putting $\EE$ in degree  
$-\dim_\CC X$ (this is denoted $\EE[\dim_\CC X]$). \\

Given 
a complex-analytic variety  $Z$, then a \emph{Hodge module} $M$ on $Z$ has two basic  ingredients, the first
being  a particular 
type of element $\MM$ of the derived category $D_c^b(Z,\QQ)$ of constructible $\QQ_Z$-modules with bounded constructible cohomology,  called a \emph{perverse sheaf}.  A form of the 
Riemann-Hilbert correspondence asserts that its complexification  $\MM_\CC\in D_c^b(Z,\CC)$ is representable 
by a regular holonomic  $\Dscr_Z$-module $\Mcal$. The  second consists of a filtration  on (the de Rham  resolution of) $\Mcal$ and should be regarded as a way of representing the 
Hodge filtration. This puts for any open subset $U\subset Z$, a Hodge filtration on $\Hl^k(U;\MM_\CC)$,  
but here again, its meaning is unclear in general.  Yet, as we will see below,  there is something of interest to say  
if we let $U$ run over a  neighbourhood basis  of a fixed point in $Z$. It should be clear from this sketchy  description that  the complexification $M_\CC$ of a Hodge module is entirely given by the filtered object $\Mcal$. 
We denote the category  of Hodge modules on $Z$ by  $\MF (Z,\QQ)$ and its complexification  (so as a quotient of $\MF (Z,\QQ)$) by  $\MF (Z,\CC)$. 

The notion  of a \emph{polarization} of a Hodge module  is couched in similar terms (see \cite{saito:hmodules}) 
and a Hodge module is called \emph{polarizable} if one exists. 
The polarizable Hodge modules are the  objects of  an abelian category $\MF (Z,\QQ)^p$ and its complexification by  $\MF (Z,\CC)^p$.  

The polarizable Hodge modules  have as their building blocks the  \emph{intermediate extensions} of (degree shifted) polarizable variations of 
Hodge structures. To explain, assume that we are given an irreducible subvariety $X^*\subset Z$ of complex dimension 
$m$ and a smooth subvariety $X\subset X^*$ whose complement in $X^*$  is a proper (closed) subvariety (which makes $X$ open-dense in $X^*$) and a polarized variation of 
Hodge structure $\EE$ on $X$ as before.
Writing   $j: X\subset Z$  for the inclusion, then the  intermediate extension
$j_{!*}\EE[m]$ of $\EE[m]$ is defined as  an object of $\MF(Z, \QQ)^p$. 
It gets its name from the fact that we have a natural factorization 
\[
j_!\EE[m]\to j_{!*}\EE[m]\to j_{*}\EE[m]
\]
in the category $\MF (Z,\QQ)^p$ and that it can be considered in the abelian category $\MF (Z,\QQ)^p$ as an image: it 
has the property that it has no $\MF (Z,\QQ)$-subquotients  supported by a proper subvariety of $X^*$.

Every  polarizable Hodge modules on $Z$ is isomorphic with a direct sum of polarizable Hodge modules of this type. We therefore assume 
in what follows that $Z=X^*$.

The corresponding extension of $\EE$ over $X^*$ without the shift is the  intersection complex $\Iscr\Cscr_{X^*}^\pt(\EE)$. 
One of Saito's theorems implies  that if $X^*$ is projective, then the intersection cohomology 
\[
\IH^k(X^*; \EE):=\Hl^k(X^*, \Iscr\Cscr_{X^*}^\pt(\EE))
\]
 has a natural polarizable Hodge structure of weight 
$w+k$. it also tells us that  for any locally closed subvariety $i_S: S\subset X^*$,  the cohomology  sheaves 
$R^k i^*_S\Iscr\Cscr_{X^*}^\pt(\EE))$ resp.\  $R^k i^!_S\Iscr\Cscr_{X^*}^\pt(\EE))$ are variations of mixed Hodge structure
over an smooth open-dense subset of  $S$

\subsection*{Representation by square integrable forms}
Let $\EE$ and $j:X\subset X^*$ be as above and suppose that $X$ is endowed with a K\"ahler metric. 
The polarization and the Hodge star  operator define an anti-linear map 
\[
\star: \Ascr^{p,q}_X(\EE_\CC)\to \Ascr^{m-p,m-q}_X(\EE_\CC^\vee)
\]
such that for every open $U\subset X$, the hermitian pairing 
\[
\la\;,\; \ra: (\alpha, \beta)\in \Ascr^{p,q}_X(U,\EE_\CC)\times \Ascr^{p,q}_X(U,\EE_\CC)\mapsto  
(\sqrt{-1})^{w+m} \alpha\cup\star\beta\in \Ascr^{m,m}_X(U),
\]
is positive with respect to the orientation (here the expression 
$\alpha\cup\star\beta$ combines  the cup product and the natural pairing of $\EE_\CC$ with its dual).
This defines a notion of square integrability: for
$\alpha\in \Ascr^k_X(U,\EE_\CC)$ is  \emph{square integrable} if $\la \alpha, \alpha\ra$ is integrable over $U$.  
The  cohomology of the  subcomplex of  $\Ascr^\pt(X,\EE_\CC)$  of forms that are together with their image under 
$\nabla$ square integrable, denoted 
$\Hl^\pt_{(2)}(X, \EE_\CC)$, is what is called the \emph{square integrable cohomology} of $\EE_\CC$. 
Zucker observed  in \cite{zucker:warped} that  if the K\"ahler metric on $X$ is complete and the square integrable 
cohomology of $\EE_\CC$ is finite dimensional, then the classical Hodge theory remains valid in this (possibly) noncompact setting. 
In particular, $\Hl^\pt_{(2)}(X, \EE_\CC)$ is 
harmonically represented and the  bigrading subsists and defines a Hodge structure on $\Hl^k_{(2)}(X, \EE_\CC)$ of weight $k+w$. So the spectral sequence for the Hodge filtration
\[
\Hl^q_{(2)}(X,\Gr_\Fcal^p \Omega_X^\pt(\Ecal))\Rightarrow  \Hl_{(2)}^\pt(X, \EE_\CC),
\]
where $\Gr_\Fcal^p \Omega_X^\pt(\Ecal)=\oplus_{p'+p''=p} \Omega_X^{p'}(\Gr_\Fcal^{p''}\Ecal)$,
degenerates  and endows $\Hl_{(2)}^\pt(X, \EE_\CC)$ with its Hodge filtration.

A presheaf complex on $X^*$ is defined by  assigning to an open subset $U$ of ${X^*}$ the subcomplex
of $\Ascr^\pt(U\cap X,\EE_\CC)$ consisting of forms that together with their image under $\nabla$ are square integrable.
Its sheafication gives   a subcomplex of $ j_* \Ascr_X^\pt(\EE_\CC)$ 
which we shall denote by $\Lscr^\pt_{{X^*},(2)}(\EE_\CC)$.
This complex comes filtered by the Hodge filtration:
\[
\textstyle \Fcal^p\Lscr^\pt_{{X^*},(2)}(\EE_\CC)=\Lscr^\pt_{{X^*},(2)}(\EE_\CC)\cap j_* \Fcal^p\Ascr^{\pt}_X(\EE_\CC).
\]
Our goal is to give conditions which imply  that this Hodge filtered complex is an incarnation of  
$\Iscr\Cscr_{X^*}^\pt(\EE_\CC)\in \MF(X, \CC)^p[-m]$.  The following is, in view of Zucker's  
harmonic representation,  merely a tautology. 

\begin{observation}\label{obs}
Suppose that ${X^*}$ is complex projective,  the K\"ahler metric is complete and that  $\Lscr^\pt_{{X^*},(2)}(\EE_\CC)$ 
is a complex of fine $\Ocal_{X^*}$-modules.
If the complex endowed with its Hodge filtration represents  $\Iscr\Cscr_{X^*}^\pt(\EE_\CC)$ as an element of 
$\MF(X, \CC)^p[-\dim_\CC X]$, then  the isomorphism $\Hl^k_{(2)}(X, \EE_\CC)\cong \IH^k({X^*}, \EE_\CC)$ takes the 
Hodge decomposition of $\Hl^k_{(2)}(X, \EE_\CC)$ defined by its harmonic representation  to the Hodge 
decomposition of $\IH^k({X^*}, \EE_\CC)$  defined by the Hodge module interpretation.
\end{observation}

\begin{remark}\label{rem:concrete}
This assumption regarding $\Lscr^\pt_{{X^*},(2)}(\EE_\CC)$ is quite strong. To see what it amounts to at a given point  $s\in X^*$, 
choose an open neighborhood $U$ of $s$ in $X^*$  which is  topologically an open cone with vertex $s$ such that $\IH^k(U,\EE)$ represents the 
 local intersection cohomology $\Hl^k(i_s^*\Iscr\Cscr_{X^*})$, where $i_s: \{s\}\subset X^*$. 
 Let  $\{ U_n\}_{n=0}^\infty$ be a strictly decreasing   open neighborhood basis of $s$ in $U=U_0$  
 compatible with the cone structure, so as to 
ensure that  the  restriction maps $\IH^k(U; \EE)\to \IH^k(U_n; \EE)$  and 
$\IH^k(U\ssm \{s\}; \EE)\to \IH^k(U_n\ssm \overline U_{n+1}; \EE)$
are isomorphisms.  

If $\Lscr^\pt_{{X^*},(2)}(\EE_\CC)$ represents $\Iscr\Cscr_{X^*}^\pt(\EE_\CC)$,  then we can then identify   $\Hl^k_{(2)} (U_n, \EE_\CC)$ with 
$\IH^k(U_n; \EE_\CC)=\Hl^k(i_s^*\Iscr\Cscr_{X^*}^\pt(\EE_\CC))$ in a way that is compatible with the restriction maps.  
By the  theory of Hodge modules, $\Hl^k(i_s^*\Iscr\Cscr_{X^*}^\pt(\EE))$ comes with a mixed Hodge structure and hence its 
complexification  has  a Hodge filtration. On the other hand, each $\Hl^k_{(2)} (U_n\cap X, \EE_\CC)$  is filtered by the  Hodge filtration on 
forms described above. If we also assume, as in the Observation \ref{obs} above, that the Hodge filtration on $\Lscr^\pt_{{X^*},(2)}(\EE_\CC)$ 
represents the one defined by the theory of Hodge modules, then the image of the Hodge filtration on  
$\Hl^k_{(2)} (U_n\cap X, \EE_\CC)$ in $ \Hl^k(i_s^*\Iscr\Cscr_{X^*}^\pt(\EE_\CC))$ under  the above isomorphism 
should then converge to the one on $\Hl^k(i_s^*\Iscr\Cscr_{X^*}^\pt(\EE_\CC))$ as $n\to\infty$.  

The same limiting  property must hold for the compactly supported version:  the Hodge  filtration on 
$\Hl^k_{(2),c} (U_n\cap X, \EE_\CC)$ should converge to the one on  $\Hl^k(i_s^!\Iscr\Cscr_{X^*}^\pt(\EE_\CC)$, as $n\to \infty$. 
 \end{remark}

In the situations we will be dealing with, $X^*$ has the following property. 

\begin{definition}\label{def:mettrivial}
We say that an analytic stratification of $X^*$ which has $X$ as its  open-dense stratum is \emph{metrically locally trivial}
if it admits local  topological trivializations  which preserve strata and  are such that when restricted to the opesn stratum,  are quasi-isometries.
\end{definition}

This property ensures that the cohomology sheaves of both the derived object $\Iscr\Cscr_{X^*}^\pt(\EE_\CC)$ and the complex $\Lscr^\pt_{{X^*},(2)}(\EE_\CC)$  are locally  trivial on each stratum. In that case the  assumption in  \ref{obs} regarding the Hodge filtration is equivalent for these two properties for  to hold for every $s\in X^*$ and 
this is in fact the definition which we will use in practice. We will need  a similar  interpretation in the situation where we only know that these hypotheses are satisfied on the complement of a point.

\begin{lemma}\label{lemma:limit}  
Assume that we are given $s\in X^*$, an open  neighborhood $U$ of $s$ in $X^*$ and a continuous  proper function 
$d: U\ssm \{s\}\to (0,\infty)$  which exhibits $(U, X\cap U)$ as an open cone with vertex $s$: $d$ extends continuously to $U$ by  mapping $s$ to $\infty$ and  $d: (U\ssm \{s\}, X\cap (U\ssm \{s\}))\to (0,\infty)$ is locally trivial with the local trivializations being on
$ X\cap (U\ssm \{s\})$ a quasi-isometry. 

If the Hodge filtered complex $\Lscr^\pt_{X^*,(2)}(\EE_\CC)$ represents $\Iscr^\pt_{X^*}(\EE_\CC)$ on  $U\ssm \{s\}$, 
then  for  $0\le n<m <\infty$  we have   natural linear  isomorphisms 
\begin{multline}\label{eqn:isos}
\Hl^k_{(2)} (d^{-1}(n,m), \EE_\CC)\cong\IH^k (d^{-1}(n,m), \EE_\CC)\cong\\
\cong  \IH^k(U\ssm \{s\}; \EE_\CC)\cong\Hl^k(i_s^*j_{s*}\Iscr\Cscr_{U\ssm\{s\}}^\pt(\EE_\CC)),
\end{multline}
(where  $j_s:U\ssm \{s\}\subset U$) and  the  Hodge filtration on  the left   defines via the above isomorphisms a family of filtrations on the right with the property that  if we first let $m\to \infty$ and then let $n\to \infty$ we get the natural  Hodge filtration on $\Hl^k(i^*_sj_{s*}\Iscr\Cscr_{U\ssm\{s\}}^\pt(\EE_\CC))$ as a limit.
\end{lemma}
\begin{proof}
Recall that  $\Iscr\Cscr_{U\ssm\{s\}}^\pt(\EE_\CC)$ stands for an object in a derived category of filtered complexes. We are given that the filtered complex $\Lscr^\pt_{U\ssm\{s\},(2)}(\EE_\CC)$ represents that object. Since  the latter   is a complex of  fine sheaves, it follows that the higher direct images $R^kd_*\Lscr^\pt_{U\ssm\{s\},(2)}(\EE_\CC)$ ($k>0$) vanish, so that 
the (zeroth) sheaf direct image $d_*\Lscr^\pt_{U\ssm\{s\},(2)}(\EE_\CC)$ is also a fine complex and represents the direct image 
$d_*\Iscr\Cscr_{U\ssm\{s\}}^\pt(\EE_\CC))$ in a  derived category of  filtered complexes on $(0, \infty)$. So the cohomology of this direct image is a graded local system on $(0, \infty)$ which we can  identify with $R^\pt d_*\Iscr\Cscr_{U\ssm\{s\}}^\pt(\EE_\CC))$. 
In particular, we  get the linear isomorphisms \eqref{eqn:isos}  with the first line even as  isomorphisms of filtered vector spaces. 
It remains to observe  that the definition of a  Hodge module is such that the Hodge  filtration on $\IH^k (d^{-1}(n,m), \EE_\CC)$ has the asserted limiting property.
\end{proof}

\section{The case of a Shimura variety}

We want to apply this to the case where  $X$ is  locally symmetric variety (and hence 
quasiprojective)  and $X^*$ the  Baily-Borel  compactification of $X$.
These data  involve a reductive linear $\QQ$-group $\Gcal$. We write $G$ for the Lie group $\Gcal(\RR)$ (\footnote{As a rule, we denote an algebraic group defined over 
$\QQ$ by a calligraphic capital font and the Lie group defined by its group of real points by its roman counterpart.}), and assume that its center $Z(G)$ is compact.  So  $G/Z(G)$ is an adjoint group whose  symmetric space $D$ is that of of  $G$. We assume that $D$ comes with a $G$-invariant complex structure, which makes it a  hermitian domain. 
We further assume given  a neat arithmetic subgroup 
$\G\subset \Gcal(\QQ)$. Then $X_\G:=\G\bs D$ is a quasi-projective complex manifold. We call such an $X_\G$ a \emph{locally symmetric variety}. We often write $X$ for $X_\G$, when $\G$ is understood.  

We also assume that we are given  a finite dimensional   $\QQ$-representation $E$ of $\Gcal$. 
This gives rise to a local system $\EE=\EE_\G$ on $X$ (given as the quotient of the trivial local system $E\times D\to D$ by the diagonal action of $\G$).
It is convenient (and for what follows no loss of generality) to assume that $\EE$ is $\QQ$-irreducible, for this implies that  $\EE$ admits the structure of a polarizable variation of Hodge structure $\EE$ over $X$ of some weight $w$ with fiber $E$  (in our set-up $E$ only determines the parity of $w$; in the Shimura setting this weight is specified by an action on $E$ of a  central extension of $\Gcal$ by a $\QQ$-split copy of $\GG_m$). We shall refer to such an $\EE$ as a \emph{locally homogeneous polarizable variation of Hodge structure} on $X$.
The  Baily-Borel compactification $j:X\hookrightarrow X^*$ completes  $X$ to a normal projective variety. Its boundary $X^*\ssm X$ comes with a stratification, all of whose members are themselves have the structure of a locally symmetric variety. More details of this construction will be recalled below.
 
 The domain  $D$ comes with a complete $G$-invariant K\"ahler metric and this makes 
$X$ a locally homogeneous  K\"ahler manifold. Zucker showed that the natural stratification of $X^*$ is metrically topologically trivial in the sense of Definition \ref{def:mettrivial} and  that each of the sheaves $\Lscr^k_{{X^*},(2)}(\EE)$ is fine (\cite{zucker:warped}, Prop.\ (4.4)). A local form of his conjecture 
(proved in different ways in \cite{looij:l2}, \cite{ss}, \cite{lr}) asserts that $\Lscr^\pt_{{X^*},(2)}(\EE)$ is an incarnation of  
$\Iscr\Cscr_{X^*}^\pt(\EE_\CC)$
in $D^b_c({X^*}, \CC)$. This implies that we have  natural isomorphism 
$\Hl^k_{(2)}(X, \EE_\CC)\cong \IH^k({X^*}, \EE_\CC)$, but as mentioned in the introduction, does not imply that 
the harmonic Hodge decomposition  equals Saito's one. 
By observation \ref{obs} this requires that we establish a \emph{filtered} local 
version of the Zucker conjecture in the sense that $\Lscr^\pt_{{X^*},(2)}(\EE)$ with its Hodge filtration is a 
cohomological Hodge complex and represents $\Iscr\Cscr^\pt_{X^*}(\EE)$ in $\MF(X^*,\CC)[-\dim_\CC X]$. The main result of this paper  says that this is the case.

\begin{theorem}\label{thm:main}
The complex $\Lscr^\pt_{{X^*},(2)}(\EE)$ is a fine $\Ocal_{X^*}$-module and  the Hodge filtered complex $\Lscr^\pt_{{X^*},(2)}(\EE)$ represents  $\Iscr\Cscr_{X^*}^\pt(\EE_\CC)$ as a degree shifted complex complexified  polarized Hodge module. So  for every open subset $U\subset X^*$ we have an isomorphism $\Hl_{(2)}^k(U\cap X, \EE)\cong \IH^k(X, \EE)$ of filtered vector spaces, which is functorial in $U$.
\end{theorem}

\subsection{Hecke equivariance} Before we proceed to give the proof, we observe the functorial behaviour in  $\G$.  A subgroup $\G'\subset \G$ of finite index determines  a natural  map $F: X_{\G'} \to X_\G$. Since $\G$ is neat, this map is a finite unramified covering  and $\EE_{\G'}$ can be identified with $F^*\EE_\G$. In particular 
there is a natural map $F_*\EE_{\G'}=F_! F^*\EE_{\G}\to \EE_{\G}$. The covering $F$ extends uniquely  to a finite projective morphism $F: X^*_{\G'} \to X^*_\G$ and so we get natural morphisms
\[
F^*\Iscr\Cscr^\pt_{X_\G^*}(\EE_\G)\to \Iscr\Cscr^\pt_{X_{\G'}^*}(\EE_{\G'}) \quad \text{and}\quad 
F_!\Iscr\Cscr^\pt_{X_{\G'}^*}(\EE_{\G'})\to \Iscr\Cscr^\pt_{X_\G^*}(\EE_\G).
\]
The first arrow is an isomorphism and the composite  map
\[
\Iscr\Cscr^\pt_{X_\G^*}(\EE_\G)\to F_*F^*\Iscr\Cscr^\pt_{X_\G^*}(\EE_\G)=F_!F^*\Iscr\Cscr^\pt_{X_\G^*}(\EE_\G)\to \Iscr\Cscr^\pt_{X_\G^*}(\EE_\G)
\]
is simply multiplication by the degree of the covering. After a degree shift, this  little diagram is  one of polarized Hodge modules.
The covering $F$ is also a local isometry and hence $\Lscr^\pt_{{X_\G^*},(2)}(\EE_\G)$ and
$\Lscr^\pt_{{X_{\G'}^*},(2)}(\EE_{\G'})$ possess the same  properties  and induce maps between the associated  spaces of harmonic forms. In particular, the validity of  Theorem \ref{thm:main} for $\G'$ implies its validity  for $\G$. Indeed, that  theorem  can be phrased 
(and is probably   best understood) in  adelic language. 

This also makes it clear 
that  we have at our disposal an action of the Hecke algebra on both  $\Iscr\Cscr_{X^*}^\pt(\EE_\CC)$ and  $\Lscr^\pt_{{X^*},(2)}(\EE)$.   Concretely, given $g\in \Gcal(\QQ)$, then $g\G g^{-1}\cap \G$ is of finite index in 
$\G$ and $g^{-1}\G g$,  and hence gives rise to a finite (Hecke) correspondence
\[
T_g: X_\G\xleftarrow{F} X_{\G\cap (g^{-1}\G g)} \stackrel{g.}{\cong}X_{(g\G g^{-1})\cap \G}\xrightarrow{F'} X_\G
\]
where $F$  and $F'$ are  the natural maps and  the middle isomorphism is induced by the action of $g$ on $D$. This  gives 
us endomorphisms $T_g^*$  of the polarized Hodge module $\Iscr\Cscr_{X^*}^\pt(\EE_\CC)$ (where the polarization gets multiplied by 
 the degree of this correspondence) and  the  complex $\Lscr^\pt_{{X^*},(2)}(\EE)$ with the latter inducing an  endomorphism of the space  harmonic forms. 
 This is relevant here, because this leads us  to impose natural boundary  conditions (namely equivariance with 
 respect to local Hecke operators) on the complex $\Lscr^\pt_{{X^*},(2)}(\EE)$  which respect the Hodge filtration and which are inherited by the harmonic forms.  These boundary conditions  were key in our original proof of the Zucker conjecture \cite{looij:l2}.

\subsection{A brief review of the Baily-Borel compactification}
What follows can be found in \cite{amrt}, but is presented here in a somewhat  more geometric  spirit, as for example  in \cite{chen_looij}.

We recall that $X^*$ appears as an orbit space of the Baily-Borel-Satake extension $D^*$ of the symmetric domain $D$.
That extension is a disjoint union of hermitian domains endowed with the so-called horocyclic topology. It contains $D$ as an open subset and $D^*\ssm D$ consists of the \emph{rational boundary components} of $D$. The closure of each rational boundary component in  $D^*$ yields its own Baily-Borel-Satake extension. The action of $G$ on $D$ does not extend to $D^*$, but the action of $\Gcal(\QQ)$  does.

The rational boundary components are in bijective correspondence with the maximal proper $\QQ$-parabolic subgroups of 
$\Gcal$. We first recall  how this comes about. 
 
Let  $\Pcal\subsetneq \Gcal$ be a maximal proper $\QQ$-parabolic subgroup. The Lie group $P$ underlying $\Pcal(\RR)$ acts transitively on $D$.  The unipotent radical $\Rcal_u(\Pcal)$ is nontrivial, but at most 2-step nilpotent:
the  center $\Ucal_\Pcal$ of $\Rcal_u(\Pcal)$ is nontrivial and $\Vcal_\Pcal:=\Rcal_u(\Pcal)/\Ucal_\Pcal$ is abelian. So these are vector groups and the commutator map produces a surjective map $\wedge_\RR^2\lie (V_\Pcal)\to \lie (U_\Pcal)$.
The reductive (Levi) quotient $\Lcal_\Pcal:=\Pcal/\Rcal_u(\Pcal)$ acts on $\lie (V_\Pcal)$ and $\lie (U_\Pcal)$ and the commutator map is equivariant for these actions.   

The center of $\Lcal_\Pcal$ contains a distinguished  $\QQ$-split copy 
$\Acal_\Pcal$ of $\GG_m$. This copy comes  with a (character) isomorphism $\chi_\Pcal: \Acal_\Pcal\cong \GG_m$  such that its  action on 
$\lie(\Ucal_\Pcal)$ resp.\  $\lie(\Vcal_\Pcal)$ is with  the character $\chi_\Pcal^2$ resp.\ $\chi_\Pcal$
We denote the  inverse of  $\chi_\Pcal$ (when regarded as a morphism) by $\alpha_\Pcal:\GG_m\cong  \Acal_\Pcal\subset \Lcal_\Pcal$.

The adjoint action  $\Lcal_\Pcal$ on $\Ucal_\Pcal$ leads to an almost  decomposition  (in the sense that intersections of factors are finite and central) 
\[
\Lcal_\Pcal=\Mcal^\ell_\Pcal \Acal_\Pcal\Mcal^h_\Pcal,
\]
 where $\Mcal^h_\Pcal$ acts trivially on $\Ucal_\Pcal$, whereas  $\Mcal^\ell_\Pcal $ acts  on $\Ucal_\Pcal$ with finite kernel. 
 The  Lie group $M^h_\Pcal=\Mcal^h_\Pcal(\RR)$ has compact center and its symmetric space is in a natural manner a hermitian domain $D_\Pcal$ (the complex structure is preserved by $M^h_\Pcal$): this is the  boundary component associated with $\Pcal$ as mentioned above.  
The complex manifold   $D_\Pcal$  is also charactrized by the property that there is  a natural $P$-equivariant holomorphic  surjection 
 \[
 \pi_\Pcal: D\to D_\Pcal
 \]
whose fibers are  the $M^\ell_\Pcal A_\Pcal^\circ$-orbits, where  $A_\Pcal^\circ$ is the identity component of $A_\Pcal=\Acal_\Pcal(\RR)$ (so  $A_\Pcal^\circ$ is isomorphic with the Lie group $\RR_{>0}$, the positive reals regarded as a multiplicative group). The group $A_\Pcal^\circ$  is also responsible for  the  geodesic flow in the fibers of $\pi_\Pcal$ towards $D_\Pcal$: for every $z\in D$, there is a unique Lie lift $\alpha_{\Pcal,z}: \RR_{>0}\to P$ of $\alpha_\Pcal:\RR_{>0}\cong A_\Pcal^\circ$ whose orbit through $z$ is a geodesic which   has  $\pi_\Pcal(z)$ as its limit at $+\infty$.  

A central role in the definition of the horocyclic topology is played by a nonempty strictly  convex self-dual cone $C_\Pcal$
in the Lie algebra $\lie(U_\Pcal)$. This cone is homogeneous: it is an $M^\ell_\Pcal$-orbit. There is a natural $P$-equivariant projection $\rho_\Pcal: D\to C_\Pcal$.  We combined  this with $\pi_\Pcal$ to get 
 a $P$-equivariant map 
\begin{equation}\label{eqn:nilfactor0}
(\rho_\Pcal,\pi_\Pcal): D\to C_\Pcal\times D_\Pcal.
\end{equation}
As mentioned, $P$ acts transitively on $D$ and it is clear that it acts on $C_\Pcal\times D_\Pcal$ via $L_\Pcal$. In fact,
the map \eqref{eqn:nilfactor0} is a principal $R_u(\Ucal_\Pcal)$ bundle. This also shows  that the complex  fiber dimension of $\pi_\Pcal: D\to D_\Pcal$ is $\dim\lie (U_\Pcal)+\half\dim\lie (V_\Pcal)$ (and that  $\dim\lie (V_\Pcal)$ must be even).

 Let $C^+_\Pcal$ stand for  the convex hull of the rational points of  the  closure  of $C_\Pcal$ in $\lie(U_\Pcal)$ (or equivalently, in  $\lie(G)$, because $\lie(U_\Pcal)\subset \lie(G)$ is subspace defined over $\QQ$). 
 This is a convex cone whose  faces are also convex cones (and open in their linear span). Any such face not equal to the vertex 
 $\{0\}$ is of the form $C_{\Pcal'}$ for some maximal proper rational parabolic group $\Pcal'$.   
We shall write  $D_\Pcal\le D_{\Pcal'}$  if $C_{\Pcal'}$ is a face of $C^+_\Pcal$ (written as $C_{\Pcal'}\le C_{\Pcal}$; it is therefore reasonable to  think 
 of the vertex $\{0\}$ as the face  associated to $\Gcal$ or to $D$).
 Let us note here that since the groups $A^\circ_\Pcal$ and $A^\circ_{\Pcal'}$ both act as scalar multiplication in $U_{\Pcal'}$, they can be naturally identified.

So the union  of $D$ and the rational boundary components  $\ge D_\Pcal$ is a subset $\Star(D_\Pcal)$ of $D^*$  whose constituents are indexed by the faces of  $C^+_\Pcal$. The projection $(\rho_\Pcal,\pi_\Pcal): D\to D_\Pcal$ extends to $(\rho_\Pcal,\pi_\Pcal): \Star(D_\Pcal)\to C^+_\Pcal\times D_\Pcal$. 
\\

The subgroup  $\G_\Pcal$ of   $\Pcal$ is arithmetic (and hence a lattice in $P$) and acts on $D_\Pcal$ 
via an arithmetic  quotient group whose  orbit space is a stratum  $S$ of $X^*$. 
We denote the kernel of this $\G_\Pcal$-action  by $\G_\Pcal^\ell\subset \G_\Pcal$.

We recall from \cite{amrt}  that a \emph{core} in $C_\Pcal$  is a 
subset  $K\subset C_\Pcal$ which is invariant under the translations in $C_{\Pcal^\ell}$ and  for which there exist $p_0, p_1\in C_\Pcal$  such that $\G_{\Pcal}^\ell (p_0+C_\Pcal)\subset K\subset 
\G_{\Pcal}^\ell (p_0+C_\Pcal)$. This notion  does not depend on $\G$  and  the convex hull of $K$ is then also a \emph{core} (see also \cite{looij:cone}).
In this paper, we shall use this notion in a more restricted sense:  we call  a core $K\subset C_\Pcal$ a \emph{$\G_\Pcal^\ell$-core} if in addition  $K$ is a  
$\G_\Pcal^\ell$-invariant, open, convex and has  smooth boundary. It is clear that $\lambda K=\alpha_\Pcal(\sqrt{\lambda})K$ is then also a core for every $\lambda >0$. 

A relative version is defined in  an obvious manner: a $\G_\Pcal$-\emph{core}  is an open $\G_\Pcal$-invariant subset $\Kcal$ of $C_\Pcal\times D_\Pcal$ with smooth boundary which meets each  fiber over $D_\Pcal$ in a $\G_\Pcal^\ell$-core.  So $\Omega_\Kcal:=(\rho_\Pcal,\pi_\Pcal)^{-1}\Kcal$ is then an open subset of $D$ which is  invariant under  both the action of the nilpotent Lie group $R_u(\Pcal)$ and the geodesic flow towards $D_\Pcal$. We extend it to a subset $\hat\Omega_\Kcal$ of $\Star(D_\Pcal)$ by stipulating that 
$\hat\Omega_\Kcal\cap D_{\Pcal'}=\pi_{\Pcal'}(\Omega_\Kcal)$ when $D_\Pcal'\ge D_\Pcal$. Then  $\hat\Omega_\Kcal$   is also invariant  under $R_u(\Pcal)$ and the geodesic flow towards  $D_\Pcal$.  

It is a fundamental  fact (and in a sense the  reason for introducing cores in the first place) that there exists  a $\G_\Pcal$-core $\Kcal$ such that every  $\G_\Pcal$-orbit in  $\Omega_\Pcal$ 
is the intersection of $\hat\Omega_\Pcal$ with a   $\G$-orbit. In other words, the orbit set $\hat U_S=\G_\Pcal\bs\Omega $ 
can then be regarded as a subset of $X^*$. This  leads us to the \emph{horocyclic topology} on $D^*$: 
it  is  the coarsest topology which  for every $\Pcal$ makes the subsets $\{\hat\Omega_\Kcal\}_\Kcal$ a basis of neighborhoods 
of $D_\Pcal$ in $D^*$ and the  projection $\pi_\Pcal: \Star(D_\Pcal)\to D_\Pcal$ continuous. The justification is that if we  endow 
$X^*$ with the quotient topology, it becomes a locally compact Hausdorff space and that the sheaf on $X^*$ of continuous 
$\CC$-valued functions which are analytic on each stratum endows $X^*$ with the structure of a normal analytic variety

\begin{remark}
The  space $C^+_\Pcal$ also admits a horocyclic  topology (this uses the self-dual properties of its faces, see \cite{looij:cone}); it is 
$\Pcal^\ell(\QQ)$-equivariant 
and makes  $\G_\Pcal^\ell\bs C^+_\Pcal$ a stratified locally compact  Hausdorff space. There is a natural extension of $\rho_\Pcal$ to 
a continuous map $\Star(D_\Pcal)\to C^+_\Pcal$ inducing a map $\G_\Pcal\bs \Star(D_\Pcal)\to \G_\Pcal^\ell\bs C^+_\Pcal$ over $S$ of stratified spaces. We will however not need this in what follows.
\end{remark}

\begin{definition}\label{def:}
We say that a  neighborhood of $D_\Pcal$ in $\Star (D_\Pcal)$  is \emph{$\G$-regular} if is  of the form  $\hat\Omega_\Kcal$
with $\Kcal$  a  $\G_\Pcal$-core and is such that for some $\eps\in (0,1)$ a nonempty  intersection of $\hat\Omega_{(1-\eps)\Kcal}$ with a 
 $\G$-orbit is in fact a $\G_\Pcal$-orbit. 

We call a   neighbourhood of $S$ in $X^*$ which appears as the image of a $\G$-regular neighbourhood of $D_\Pcal$ in $\Star (D_\Pcal)$ a \emph{fundamental regular neighbourhood} of $S$ in $X^*$. 
\end{definition}

A fundamental regular neighbourhood $\hat U_S=\G_\Pcal^\ell\bs\hat\Omega_\Kcal$ of  $S$ in $X^*$ has several nice properties. First of all,  it  is invariant under a geodesic flow  towards $S$. To be precise,  it takes $\hat U_S$ to $\mu\hat U_S:=\G_\Pcal\bs \Omega_{\mu\Kcal}$ for every $\mu\ge 1$. This 
defines a homeomorphism $(0,\infty)\times \p\hat U_S\cong \hat U_S\ssm S$ which extends to a stratified homeomorphism of $\hat U_S$ onto the open mapping cone of $\pi_S|\p\hat U_S$ (which is here defined as the quotient of   $(0, \infty]\times  \p\hat U_S$ obtained by collapsing  $\{\infty\}\times  \p\hat U_S\cong \p\hat U_S$ along $\pi_S$). Secondly, 
the projection $\pi_S:\p\hat U_S\to S$ is topologically locally trivial in a strata preserving manner. And thirdly, 
$\pi_S: U_S:=\G_\Pcal\bs\Omega_\Kcal\to S$  factors  over a fibration  into compact  nilmanifolds
\begin{equation}\label{eqn:nilfactor}
 U_S\xrightarrow{\rho_S} \G_{\Lcal_\Pcal}\bs\Kcal\xrightarrow{\pi'_S} S,
\end{equation}
where $\G_{\Lcal_\Pcal}$ is the image of $\G_\Pcal$ in $L_\Pcal$. The fiber of $\rho_S$  through any  point of $ U_S$ is naturally 
identified with the nilmanifold $R_u(\Pcal)/\G_{R_u(\Pcal)}$ (this is a torus bundle over a torus; in particular, $\rho_S$ is proper). The fiber 
of $\pi'_S$ over $s\in S$ is the $\G_\Pcal^\ell$-orbit space of  the  core $\Kcal_s$. 
There exist metrically locally trivial trivializations of $\hat U_S\to S$ in the sense of Definition \ref{def:mettrivial} compatible with the 
geodesic flow towards $S$.

\subsection{Proof of the Theorem \ref{thm:main}}
We begin with reviewing the proof of the Zucker conjecture as given in \cite{looij:l2} and \cite{lr}.
Every $g\in \Pcal^\ell(\QQ)$ induces  Hecke correspondence $T^\Pcal_g$ of the germ of $X^*$ at $S$, which is proportional to the identity over $S$. It is a local isometry which lifts to $\EE$. So if  $i_S: S\subset  X^*$ is the inclusion, then $T^\Pcal_g$
induces an endomorphism $T^\Pcal_g{}^*$ of the mixed Hodge module $i_S^*j_{!*}\EE$, the intersection complex $ i_S^*\Iscr\Cscr_{\Omega}^\pt(\EE)$ and  $i_S^*\Lscr^\pt_{\Omega \ssm S,(2)}(\EE)$.

In case  $g$ is such that  $g\G_\Pcal g^{-1}\subset\G_\Pcal$, then it is the identity over $S$ and we can choose the $\G$-core $\Kcal$ such that  $g$ maps a $\Kcal$ to itself and hence defines in fact a finite map from a fundamental regular neighborhood $\hat U_S$ as above onto neighborhood of $S$ contained in $\hat U_S$. 

For  example, if we take $g\in \Rcal_u(\Pcal)(\QQ)$, then $T^\Pcal_g$  preserves each  fiber of $\rho_S$. Insofar the action 
on  differential forms  is concerned, we may take these forms in such a manner that their pull-back to 
$D$ is invariant under the nilpotent  Lie group $R_u(\Pcal)$ (the harmonic forms will automatically have this property). 
The  Hecke action of $T^\Pcal_g{}^*$ on such forms is trivial. This explains  why  for most of its uses, $T^\Pcal_g{}^*$ 
only depends on the image $\bar g$ of $g$ in the Levi quotient. 

Of special interest here is the case when 
$\bar g\in \Acal_\Pcal(\QQ)$.  If $\chi_\Pcal(\bar g)=\sqrt{q}$ (or equivalently, $\alpha_\Pcal(\sqrt{q})=\bar g$), then $g\G_\Pcal g^{-1}\subset\G_\Pcal$, provided $q$ is a  positive  square  such that  $q-1$ is sufficiently divisible.  We assume this is the case and then  write $T^S(q)$ for $T^\Pcal_g$. The  action of $T^S(q)$  on  $ U_S$ is geometrically easy to understand in terms of the 
factorization \eqref{eqn:nilfactor}: since $g$ acts as multiplication by $q$ resp.\  $\sqrt{q}$ on $\lie(U_\Pcal)$ resp.\  $\lie(V_\Pcal)$, the map $T^S(q)$
takes $[\G_\Pcal^\ell y]\in \G_\Pcal^\ell\bs\Kcal$ to $[\G_\Pcal^\ell qy]\in \G_\Pcal^\ell\bs\Kcal$ with the map on their nilmanifold fibers being an  isogeny: recall that  the nilmanifold is a torus bundle and that over a torus and $T^S(q)$ acts in the base torus as 
multiplication by $\sqrt{q}$ and in the fiber torus as  multiplication by $q$. So  $T^S(q)$ induces on the germ of $X^*$ at  $S$ a 
morphism of degree 
\[
q^{\dim\lie(U_\Pcal)} (\!\sqrt{q})^{\dim\lie(V_\Pcal)}=q^{\dim\lie(U_\Pcal)+\half\dim\lie(V_\Pcal)}=q^d,
\]
where $d$ is  the  complex codimension of $S$ in $X^*$. The  geodesic flow \emph{away} from $S$ produces an isotopy of $T^S(q)$ with  a self map of $\hat U_S$ which is the identity over $\G_{\Lcal_\Pcal}\bs\Kcal$. This self map then induces in each
in each fiber over $\G_{\Lcal_\Pcal}\bs\Kcal$ an isogeny as above.  Both the singular  cochain complex  and the $\EE$-valued de Rham complex of  the  fiber through a given point are naturally quasi-isomorphic  with the  Lie algebra  subcomplex  
\begin{equation}\label{eqn:}
\Hom_\RR( \wedge^\pt\lie(R_u(\Pcal), E).
\end{equation}
The action of $T^S(q)$  is in this subcomplex  given by the lift $g\in \Pcal(\QQ)$ of $\bar g=\alpha_\Pcal(\sqrt{q})$ and hence is semisimple with eigenvalues  powers of $\sqrt{q}$. This is of course merely the evaluation of the action of a one-parameter subgroup 
on this Lie algebra complex. In this form  it already appeared earlier in the work of Borel  \cite{bw} and  Zucker \cite{zucker:l2_ii}. 
It shows that the total direct image of $\EE|U_S$ on $\G_{\Lcal_\Pcal}\bs \Kcal$ decomposes in the derived category 
according to the integral powers of $\sqrt{q}$. This implies the semisimplicity of $T^S(q)$-actions on various cohomology groups. 

Let $s\in S$, denote by $U_s\subset \hat U_s$  the fiber of  $U_S\subset \hat U_S$ over $s$ and put $\EE_s:=\EE|U_s$. By the  defining properties of the intersection complex,
\[
\IH^k(\hat U_s, \EE_s )=
\begin{cases}
\IH^k(\hat U_s\ssm \{s\}, \EE_s)&\text{if $k<d$},\\
0 & \text{if $k\ge d$.}
\end{cases}
\]
\[
\IH_c^k(\hat U_s, \EE_s)=
\begin{cases}
\IH_c^{k}(\hat U_s\ssm \{s\}, \EE_s)&\text{if $k>d$},\\
0 & \text{if $k\le d$}
\end{cases}
\]
and the  polarization pairing  on $\EE_s$ induces the Poincar\'e duality pairing
\[
\IH^k(\hat U_s\ssm\{s\}, \EE_s)\times \IH_c^{2d-k}(\hat U_s\ssm \{s\}, \EE_s)\to \IH_c^{2d}(\hat U_s\ssm \{s\}; \QQ) (\cong \QQ).
\] 
The pair $(\hat U_s, \EE_s)$ comes with an endomorphism $T^s(q)$ of degree $q^d$ which acts semisimply on $\IH^\pt(\hat U_s\ssm\{s\},\EE_s)$ with eigenvalues powers of  $\sqrt{q}$.  The pairing  is $T^s(q)$-equivariant, where $T^s(q)$  acts on $\IH^{2d}(\hat U_s\ssm \{s\}; \QQ)$  as multiplication by $q^d$. 
We  restate this  in  terms of (a derived category of) sheaf complexes: letting $j_S$ stand for the inclusion $\hat U_S\ssm S\subset \hat U_S$, then
\begin{equation}\label{eqn:dec}
i_S^*j_{S*}\Iscr\Cscr_{\hat U_S\ssm S}^\pt(\EE)= i_S^*\Iscr\Cscr_{\hat U_S}^\pt(\EE)\oplus  i_S^!\Iscr\Cscr_{\hat U_S}^\pt(\EE)[-1],
\end{equation}
with the first summand living in degrees $<d$ and the second in degrees $>d$. This comes a perfect duality pairing (in a derived category)
\[
i_S^*j_{S*}\Iscr\Cscr_{\hat U_S\ssm S}^\pt(\EE)\otimes_{\QQ_S}  i_S^!j_{S*}\Iscr\Cscr_{\hat U_S\ssm S}^\pt(\EE)\to \QQ_S(-d)[-2d].
\]

The following is proved in \cite{looij:l2} (Prop.\  3.8) (see also \cite{lr}, Thm.\  3.2). The semisimplicity of the  action of  $T^S(q)$ on $R^\pt i_S^*j_{S*}\Iscr\Cscr_{\hat U_S\ssm S}^\pt(\EE)$  is not stated there explicitly, but follows from the preceding discussion.

\begin{proposition}\label{prop:Phi_IC}
The action of the local Hecke operator $T^S(q)$  on  the local system $R^\pt i_S^*j_{S*}\Iscr\Cscr_{\hat U_S\ssm S}^\pt(\EE)$ is semisimple  with eigenvalues that are integral powers of $q$ and the summand associated with the eigenvalue $q^l$  is a locally homogeneous polarized Hodge structure on $S$ of weight $w+\ell$. 

Furthermore,  $R^\pt i^*_S\Iscr\Cscr^\pt_{\hat U_S}(\EE)$ resp.\ 
$R^\pt i_S^!\Iscr\Cscr_{\hat U_S}^\pt(\EE)[-1]$ is  the subsum of the summands with eigenvalue  $<q^d$ resp. $>q^d$
(so the  eigenvalue $q^d$ does not occur).
\hfill $\square$
\end{proposition}

The proof of  our main theorem will also use the following simple linear algebra observation, stated here as a lemma.

\begin{lemma}\label{lemma:grassmannlimit}
Let  $T$ be  semisimple transformation of a finite dimensional $\CC$-vector space $V$ whose eigenvalues are integral powers of $q$. Then for every subspace  $F\subset V$  the limit $F_\infty:=\lim_{n\to \infty} T^nF$ (taken in the Grassmannian of $V$) is of the form
$\oplus_\ell F_{\infty, \ell}$ with $F_{\infty, \ell}$ is contained in the $q^\ell$-eigenspace $V_\ell$ of $T$ and obtained  
as  the image of  $F\cap \oplus_{\ell'\le \ell} V_\ell'$ under the projection onto $V_\ell$.\hfill $\square$
\end{lemma}

\begin{proof}[Proof of Theorem \ref{thm:main}]
Let us make the inductive assumption that  it has been established that the Hodge filtered complex $\Lscr^\pt_{X^*(2)}(\EE)$ represents $\Iscr\Cscr^\pt(\EE)$ on the union of strata of complex codimension $<d$.  
The induction step consists in proving that this is then also true along  a codimension $d$ stratum $S$ as above. 

Our induction assumption implies that the complex  $i^*_Sj_{S*}\Lscr^\pt_{\hat U_S\ssm S,(2)}(\EE)$ represents  $i^*_Sj_{S*}\Iscr\Cscr_{\hat U_S\ssm S}^\pt(\EE)$.  As mentioned, Zucker proved that  $\pi_S:\hat U_S\to S$ admits metric local trivializations in the sense of Definition \ref{def:mettrivial} that are 
compatible with the geodesic flow towards $S$ (the Theorem in \S 3 of \cite{zucker:l2_ii}). So  
 $\Lscr^\pt_{X^*,(2)}(\EE)$ is locally constant along $S$ and the  cohomology  of the stalk of 
$i^*_Sj_{S*}\Lscr^\pt_{\hat U_S\ssm S,(2)}(\EE)$ at $s\in S$  is also that  of its form restriction to the fiber over $s$, i.e., the stalk of 
$R^kj_{s*}\Lscr^\pt_{\hat U_s\ssm \{s\},(2)}(\EE_s)$  at $s\in S$ (where $j_s:\hat U_s\ssm\{s\}\subset \hat U_s$). We are then in a situation to which Lemma \ref{lemma:limit}  applies: it tells us that 
for  $1\le \lambda<\mu$ we have natural isomorphism  of vector spaces
\begin{multline*}
\Hl_{(2)}^\pt(\lambda \hat U_s\ssm \mu\hat U_s, \EE_s )\cong \IH^\pt(\lambda \hat U_s\ssm \mu\hat U_s, \EE_s)\cong \\\cong\IH^\pt(\hat U_s\ssm \{s\}, \EE_s)\cong \Hl^k(i_s^*j_{s*}\Iscr\Cscr^\pt_{\hat U_s\ssm \{s\}}(\EE))
\end{multline*}
with the first row an isomorphism of filtered vector spaces.
We  get the Hodge filtration on the last term  from  (the image of) the Hodge filtration of the first term by first taking the 
limit  for $\mu\to \infty$ and then the limit  for $\lambda\to \infty$.  
The endomorphism  $T^s(q)$ acts on the  last term (and preserves its Hodge filtration). 
If we  transfer this action to the top row, we  can make it geometrically explicit 
by noting that for   $1\le \lambda <\mu$,  the
 restriction map
\begin{equation}\label{eqn:restr}
\Hl_{(2)}^k(\lambda\Omega_{s}\ssm \mu\Omega_{s}, \EE_s )\to \Hl_{(2)}^k(\lambda q\Omega_{s,r}\ssm \mu q\Omega_{s}, \EE_s )
\end{equation}
is an isomorphism and that  $T^s(q)$ induces a map  in the opposite direction, so that we get  an endomorphism of $H_{(2)}^\pt(\lambda\Omega_{s,n}\ssm \mu\Omega_{s}, \EE_s )$. This is indeed  the one coming from $\Hl^k(i_s^*j_{s*}\Iscr\Cscr^\pt_{\hat U_s\ssm \{s\}}(\EE))$. It follows  that we get the limit of the Hodge filtration on the stalk of $R^kj_{s*}\Lscr^\pt_{\hat U_s\ssm \{s\},(2)}(\EE_s)$  at $s$, by taking  the limit of  the Hodge filtration on
$\Hl_{(2)}^\pt(\lambda  U_s\ssm \mu U_s, \EE_s )$ under the positive powers of $T^s(q)$.

Lemma \ref{lemma:grassmannlimit} applied to this situation tells us  that if we regard $H_{(2)}^\pt(\lambda  U_s\ssm \mu U_s, \EE_s )_\ell$
as a subquotient of  $H_{(2)}^\pt(\lambda  U_s\ssm \mu U_s, \EE_s)$ (rather than as a direct summand), then its  Hodge filtration reproduces the Hodge filtration of $\IH^\pt(\hat U_s\ssm \{s\}, \EE_s)$. In particular, it defines on 
$H_{(2)}^\pt(\lambda  U_s\ssm \mu U_s, \EE_s )_\ell$ a pure Hodge structure of weight $m+\ell$.
Since  the eigenvalue $q^d$ does not occur, Zucker's weight  computation (formula (8) of the Theorem in \S 3 of \cite{zucker:l2_ii}) allows us to conclude that the restriction map
\[
\oplus_{\ell <d} H_{(2)}^\pt(\lambda U_s\ssm q\lambda U_s, \EE_s )_\ell\to H_{(2)}^\pt(\lambda U_s, \EE_s )
\]
is an isomorphism. Hence 
\[
H_{(2)}^\pt(\lambda  U_s, \EE_s )=\oplus_{\ell <d} H_{(2)}^\pt(\lambda  U_s, \EE_s )\cong 
\oplus_{\ell <d}\IH^\pt(\hat U_s, \EE_s)=\IH^\pt(\hat U_s, \EE_s)
\]
becomes  is an isomorphism of Hodge structures if we let $\lambda\to \infty$.
In the same way we find the corresponding  property for the compactly supported version:
\[
H_{(2),c}^\pt(\lambda  U_s, \EE_s )=\oplus_{\ell >d} H_{(2),c}^\pt(\lambda  U_s, \EE_s )\cong 
\oplus_{\ell >d}\IH_c^\pt(\hat U_s, \EE_s)=\IH_c^\pt(\hat U_s, \EE_s)
\]
This completes the induction step and thereby finishes the proof of Theorem \ref{thm:main}.
\end{proof}

 \end{document}